\documentclass[11pt]{amsart}
\usepackage[margin=3cm]{geometry}
\usepackage{bm}
\usepackage{amsthm}
\usepackage{amssymb}
\usepackage{mathrsfs}
\numberwithin{equation}{section}
\usepackage{tikz}
\usepackage{tkz-euclide}
\usepackage{graphicx}
\usepackage[colorlinks=true, pagebackref]{hyperref}
\usepackage{xcolor}

\newtheorem{theorem}{Theorem}[section]
\newtheorem{definition}[theorem]{Definition}
\newtheorem{conjecture}[theorem]{Conjecture}

\newtheorem{example}[theorem]{Example}

\newcommand{\R}{\mathbb{R}}

\newcommand{\e}{\epsilon}

\newcommand{\pr}{\partial}

\DeclareMathOperator{\dist}{dist}

\newcommand{\bn}{\mathbf{n}}

\newtheorem{thm}[theorem]{Theorem}
\newtheorem*{thm*}{Theorem}
\newtheorem{lem}[theorem]{Lemma}
\newtheorem{prop}[theorem]{Proposition}
\newtheorem{coro}[theorem]{Corollary}

\theoremstyle{definition}

\newtheorem{rmk}[theorem]{Remark}

\begin{document}

\begin{abstract}
We construct new examples of immortal mean curvature flow of smooth embedded connected hypersurfaces in manifolds, which converge to minimal hypersurfaces with multiplicity $2$ as time approaches infinity.

\end{abstract}

\title{Mean curvature flow with multiplicity $2$ convergence in manifolds}

\author{Jingwen Chen, Ao Sun}
\address{Department of Mathematics, University of Pennsylvania,
David Rittenhouse Lab,
209 South 33rd Street,
Philadelphia, PA 19104}
\email{jingwch@sas.upenn.edu}
\address{Department of Mathematics, Lehigh University, Chandler-Ullmann Hall, Bethlehem, PA 18015}
\email{aos223@lehigh.edu}
\date{}
\maketitle

\section{Introduction}

The mean curvature flow is the gradient flow of the area functional. This feature makes it one of the most natural extrinsic geometric flows. Mean curvature flow is widely used to study geometry and topology, and we refer the readers to some of the applications in \cite{Grayson89_shortening, WangMT02_longtimeMCF, HaslhoferKetover19_minimalS2, BuzanoHaslhoferHershkovitz19_ModuliTori,BuzanoHaslhoferHershkovitz21_ModuliSpheres,  GuaracoLimaPallete21_MCFHyperbolic}. A central topic in the study of mean curvature flow is to understand the singular behavior and long-time behavior, and it is particularly important to the potential applications.

Among the possible exotic phenomena that may show up, higher multiplicity convergence is one of the most significant phenomena and has attracted much attention. Roughly speaking, higher multiplicity convergence means that the hypersurfaces can be decomposed into several connected components outside a small set, and each one of the connected components converges to the limit hypersurface smoothly. Recently, the Multiplicity One Conjecture proposed by Ilmanen was resolved by Bamler-Kleiner \cite{BamlerKleiner23_multiplicity}, showing that higher multiplicity can not show up as the blow-up of a singularity of a mean curvature flow of closed embedded surfaces in $\R^3$. In contrast, the authors \cite{chen2023mean} constructed an immortal mean curvature flow in $\R^3$, showing that higher multiplicity convergence can show up as the long-time behavior of mean curvature flow of connected surfaces in $\R^3$.

In this paper, we construct new examples of mean curvature flow in closed manifolds that converge to higher multiplicity minimal hypersurfaces as time goes to infinity. Recall that an immortal mean curvature flow is a smooth mean curvature flow that exists for all positive time. In the following, suppose $S^n$ equips the round metric and $S^n\times[-1,1]$ equips the standard product metric. It is not too hard to see that $S^n\times\{0\}$ is a minimal hypersurface.

\begin{theorem}
For all $n\geq 2$, there exists a smooth embedded connected immortal mean curvature flow $(M(t))_{t>0}$ in $S^n\times[-1,1]$ such that $M(t)$ converges to $S^n\times\{0\}$ with multiplicity $2$ as $t\to\infty$.
\end{theorem}

Of course, one can close up $S^n\times[-1,1]$ to get a compact manifold, or extend it to a complete noncompact manifold.

This new family of examples shows that the higher multiplicity convergence of mean curvature flow can occur for compact mean curvature flow in a closed manifold. In particular, there exists an infinite time singularity. This contrasts with a result of Jingyi Chen and Weiyong He \cite{ChenHe10_SingularTime}, which proved that in manifold with certain curvature and growth conditions, the mean curvature flow can only have finite time singularities. We note that their growth condition does not hold for $S^n\times\R$. 

Our new examples enhance our understanding of using mean curvature flow to construct minimal hypersurfaces. When higher multiplicity convergence shows up, the geometry and topology of the limit hypersurface may be different from the flow. This phenomenon has been discovered in other fields of geometric analysis, such as the Min-max theory. In \cite{WangZhou22_minmax-higher-multiplicity}, Zhichao Wang and Xin Zhou constructed an example to show higher multiplicity minimal surfaces can show up in the Almgren-Pitts min-max theory. In \cite{WangZhou23_existence4spheres}, Zhichao Wang and Xin Zhou observed that the possible obstruction to the existence of $4$ minimal spheres in $S^3$ with an arbitrary metric is the higher multiplicity convergence in the min-max theory.

Motivated by the generic multiplicity $1$ property in the min-max theory proved by Xin Zhou \cite{Zhou19_multi1}, we expect that higher multiplicity convergence is a non-generic property. 

\begin{conjecture}
Starting from a generic closed hypersurface in a Riemannian manifold, either the long-time limit has multiplicity $1$, or the hypersurface breaks into different connected components, and each one of the connected components converges to a multiplicity $1$ limit.
\end{conjecture}

We remark that the limit minimal surfaces in our examples are stable. Also motivated by the min-max theory and the resolution of the Multiplicity One Conjecture by Bamler and Kleiner \cite{BamlerKleiner23_multiplicity}, we expect that higher multiplicity convergence can not occur if the limit minimal hypersurface is unstable.

\begin{conjecture}
    Suppose $N$ is a Riemannian manifold, $(M(t))_{t>0}$ is a mean curvature flow of closed embedded hypersurfaces. if the limit of $M(t)$ as $t\to\infty$ is a minimal hypersurface $\Sigma$ with multiplicity $k$, and $\Sigma$ is unstable, then $k=1$.
\end{conjecture}

Now we discuss the strategy of the proof. The main idea is similar to \cite{chen2023mean} in which the authors used an interpolation argument. Here we also constructed a family of initial data $(M^\delta)$, which are rotationally symmetry under the $SO(n)$ action on the sphere factor, and reflective symmetry along $S^n\times\{0\}$ and the equators of $S^n$. Moreover, when $\delta$ is small, the mean curvature flow starting from $M^\delta$ would generate a neck singularity at the north pole of $S^n\times\{0\}$, and when $\delta$ is large, the mean curvature flow starting from $M^\delta$ would shrink to a point on the equator of $S^n\times\{0\}$. Then there exists an $\eta$ in between such that the mean curvature flow starting from $M^\eta$ exists for all future time, namely it is an immortal flow.

It remains to show this immortal flow converges to $S^n\times\{0\}$ with multiplicity $2$. To do so, we need to use some barriers to control the immortal flow. In \cite{chen2023mean}, we used the planes, the catenoids, and the Angenent torus as barriers. In this paper, we find analogous barriers: the spherical minimal hypersurfaces, the spherical catenoids, and $\lambda$-Angenent torus. We will discuss these special solutions in the preliminary section.

\section{Preliminary}

\subsection{Metric on \texorpdfstring{$S^n\times [-1,1]$}{Lg} and rotationally symmetric mean curvature flows}
We equip $S^n$ with the standard metric, namely the induced metric on $S^n:=\{(x_1,\cdots,x_{n+1})|x_1^2+x_2^2+\cdots +x_{n+1}^2=1\}$ as a hypersurface in $\R^{n+1}$. Then we equip $S^n\times [-1,1]$ with the product metric, denoted by $g$. We can also view $S^n \times [-1,1]$ as the hypersurface in $\R^{n+1} \times \R$, equipped with the induced metric.

We consider the $SO(n)$ action rotating the hyperplane generated by the first $n$ coordinates of $\R^{n+1}\times\R$. This action can be induced to $S^n\times[-1,1]$ naturally. Any hypersurface in $S^n\times[-1,1]$ is called {\bf rotationally symmetric} if it is invariant under this $SO(n)$ action. After taking the quotient, $S^n\times[-1,1]$ becomes $[0,\pi]\times[-1,1]$, and we use $x$ and $y$ to denote the coordinates on $[0,\pi]$ and $[-1,1]$ respectively.

Given a rotationally symmetric hypersurface $\Gamma$ in $S^n\times[-1,1]$, there is an associated curve called {\bf profile curve} $\gamma:I\to [0,\pi]\times[-1,1]$, while $I$ can be an interval or $S^1$, such that $\Gamma$ is generated by rotating $\gamma$. That says, if we write $\gamma(s)=(x(s),y(s))$, then we can parameterize $\Gamma$ as follows: \begin{align*}
    &\Gamma=\{(\cos(x(s)),\sin(x(s))\cos(\varphi_2),\cdots,\sin(x(s))\sin(\varphi_2)\cdots\sin(\varphi_{n}),y(s))\in \R^{n+1}\times\R\},
    \\
    &\text{while }(s,\varphi_2,\cdots,\varphi_{n})\in I\times[0,\pi]\times\cdots\times[0,\pi]\times[0,2\pi].
\end{align*}
The area form of this parametrization is given by 
\[\sqrt{(x')^2+(y')^2}\sin^{n-1}(x(s))\sin^{n-2}(\varphi_2)\cdots \sin(\varphi_{n-1})dsd\varphi_2d\varphi_3\cdots d\varphi_{n}.\]
After taking integration and using Fubini's theorem, we see that the area of $\Gamma$, up to a constant, equals the length of the curve $\gamma\subset [0,\pi]\times[-1,1]$ under the metric
\begin{equation}
    g_{\text{rot}}=(\sin(x))^{2(n-1)}(dx^2+dy^2).
\end{equation}

We will be interested in mean curvature flows with initial data that are rotationally symmetric, reflexive symmetric with respect to the middle sphere $S^n \times \{0\}$, and reflexive symmetric with respect to the hypersurface $\{s = \frac{\pi}{2} \} \times [-1,1]$. We call a smooth hypersurface in $S^n \times [-1,1]$ with all above symmetry conditions as a {\bf desired hypersurface}. By the short-time existence and uniqueness of the mean curvature flow, we know the rotational symmetry and reflexive symmetry are preserved under the mean curvature flow. Consider a mean curvature flow of desired hypersurfaces $(M(t))_{t \in [0,T)}$ on $S^n \times [-1,1]$. For each hypersurface $M(t)$, after taking the quotient of the $SO(n)$ action and the two reflections, we get a curve $\gamma_t$ in the region $D:= [0, \frac{\pi}{2}] \times [0,1]$, and we name $\gamma_t$ as the {\bf section curve} of $M(t)$ (note this is the top left part of the profile curve). We say $\gamma_t$ is an {\bf ascending section curve} if it is the graph of an increasing function, and let $\mathcal{S}$ denote the set of function $f$ such that the graph of $f$ is an ascending section curve of a desired hypersurface in $S^n \times [-1,1]$.  

The reflexive symmetry implies that the curve $\gamma_t$ intersects the two sides $\{x=0\}$, $\{y = 0\}$ orthogonally. We name the intersection point of $\gamma_t$ with $\{y = 0\}$ as the {\bf head point}. This curve $\gamma_t$ can be expressed as the union of two graphs of smooth functions $u$ and $v$, and we call $u$ the {\bf vertical graph function} and call its graph as the {\bf vertical graph}; we call $v$ as the {\bf horizontal graph function} and call its graph as the {\bf horizontal graph}. See the picture below for an example.

\begin{figure}[htb]
    \centering
    \tikzset{global scale/.style={
    scale=#1,
    every node/.append style={scale=#1}
  }
}
    \begin{tikzpicture}[global scale = 0.7]
    \draw (0,0) -- (9.425,0) node[below] {$(\frac{\pi}{2},0)$};
    \draw (0,0) -- (0,6) node[left] {$(0,1)$};
    \draw (9.425,0) -- (9.425,6);
    \draw (0,6) -- (9.425,6);
    \node[below] at (5,-0.1) {$x$};
    \node[below] at (-0.3,3.2) {$y$};
    \node[below] at (-0.2,0) {$(0,0)$};
    \node[below] at (9.925,6.4) {$(\frac{\pi}{2},1)$};
    \node[below] at (2,5) {region $D$};
    \draw plot[smooth,tension=.6]
coordinates {(2.5,0) (2.5, 0.5) (2.6, 1) (2.7,1.3) (3,2) (4,3) (5,3.5) (6,3.7)(7,3.85) (8, 3.95) (9,4) (9.425,4)};
    \node[below left] at (2,1.7) {vertical};
    \node[below left] at (1.9,1.2) {graph};
    \node[below left] at (2.12,0.7) {$x = u(y)$};
    \node[below] at (7,4.7) {horizontal graph $y = v(x)$};
    \node[below] at (5.1,2.8) {section curve $\gamma_t$};    \fill(2.5,0)circle(1.5pt)node[below]{head point};
    \draw (2.5,0.1) coordinate (a) -- (2.5,0) coordinate (b) -- (2.51,0) coordinate (c);
    \tkzMarkRightAngle (a,b,c)
     \draw (9.325,4) coordinate (d) -- (9.425,4) coordinate (e) -- (9.425,3.9) coordinate (f);
    \tkzMarkRightAngle (d,e,f)
    \end{tikzpicture}
    \caption{Example of a section curve that we study.}
\end{figure}
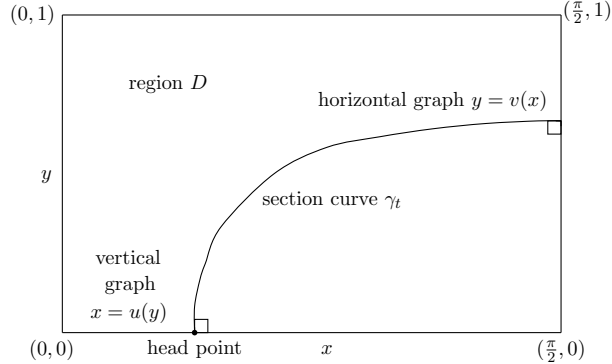

Let $u_0$ and $v_0$ be the vertical graph function and the horizontal graph function of the initial curve $\gamma_0$. Suppose after reflecting and rotating $\gamma_0$ we get a hypersurface $M(0)$ in $S^n\times[-1,1]$. Then a family of hypersurfaces $\{M(t)\}_{t \in [0,T)}$ is said to be a mean curvature flow with initial condition $M(0)$ if the corresponding vertical graph function, horizontal graph function $u(\cdot, t)$, $v(\cdot, t)$ of $\gamma_t$ satisfy the following equations (see Appendix \ref{Appendix:curves}):

\begin{equation}
\label{mcf equation} 
\left\{
\begin{aligned}
&\frac{\partial u}{\partial t}=\frac{u_{yy}}{1+(u_y)^2}-(n-1)\frac{\cos(u)}{\sin(u)}, \\
&\frac{\partial v}{\partial t} = \frac{v_{xx}}{1+(v_x)^2} + (n-1)\frac{\cos x}{\sin x}v_x, \\
&u(\cdot,0) = u_0,\  v(\cdot,0) = v_0, \\
&u_y(0, \cdot) = 0, \ v_x(\frac{\pi}{2}, \cdot) = 0.
\end{aligned} 
\right.
\end{equation}

For simplicity, we will say $f(\cdot, t)$ is a solution to equation \eqref{mcf equation} if its corresponding vertical graph function and horizontal graph function solve this equation. We will call the first equation in \eqref{mcf equation} the vertical graph equation, and call the second equation in \eqref{mcf equation} the horizontal graph equation.

\subsection{Minimal hypersurfaces as barriers} \label{minisurface barrier} Suppose $\gamma^1$ and $\gamma^2$ are two curves in $D$. We say that $\gamma^1$ is \textbf{on top of} $\gamma^2$, if for all pairs $(c, y_1, y_2)$ such that $(c,y_1) \in \gamma^1, (c,y_2) \in \gamma^2$, it holds that  $y_1 \geq y_2$. Recall the comparison principle for mean curvature flow:

\begin{prop}[Comparison principle] \label{comparison}
    Suppose $N$ is a closed manifold, $(M^1(t))_{t\in[0,T]}$ and $(M^2(t))_{t\in[0,T]}$ are two mean curvature flows of smooth embedded hypersurfaces in $N$. If $M^1(0)$ does not intersect $M^2(0)$, then $M^1(t)$ does not intersect $M^2(t)$ for all $t\in[0,T]$.
\end{prop}

Therefore, we can use barriers that are either on top of the flow or under the flow at the initial time to control the behavior of the flow at a later time. We will frequently use the following special solutions as barriers in this paper. 

Firstly, we consider the static solutions known as the minimal hypersurfaces. The simplest possible examples are the horizontal graphs $(x,v(x,t))$ given by $v(x,t)\equiv C$ for a constant $C\in[-1,1]$. Such a minimal hypersurface is just a section of $S^n\times[-1,1]$, and we call it a {\bf minimal sphere}. There is only one constant vertical graph $(u(y,t),y)$ given by $u(y,t)\equiv \pi/2$. It is the Cartesian product of the geodesic hypersphere in $S^n$ with $[-1,1]$, and we still call it a {\bf geodesic hypersphere}.

Another class of minimal hypersurfaces is the analog of the catenoids, and we call them {\bf spherical catenoids}. Given a parameter $C\in(0,\pi/2)$, we can solve the ODE 

\begin{equation}\label{eq:ODE for minimal}
    \frac{u_{yy}}{1+(u_y)^2}-(n-1)\frac{\cos(u)}{\sin(u)}=0,
\end{equation}
with initial condition $u(0)=C$ and $u_y(0)=0$. First, \eqref{eq:ODE for minimal} is invariant under $y\leftrightarrow-y$, so the solution must be an even function, and hence we only need to study the part of the solution where $y\geq 0$. It is straightforward to see that $u_{yy}(0)>0$, and hence $u_y>0$ in a neighborhood of $0$. Multiplying both sides of \eqref{eq:ODE for minimal} by $u_y$ yields the identity
\[
\frac{1}{2}(\log (1+(u_y)^2))_y-(n-1)(\log(\sin u))_y=0,
\]
and then we can conclude that there exists a constant $\lambda_C$ such that
\[
u_y^2=\lambda_C^2(\sin u)^{2(n-1)}-1.
\]
By plugging in $u(0)=C$ we obtain
\[
\lambda_C=(\sin C)^{-(n-1)}.
\]
Then we obtain that $u_y(y)>0$ for all $y\in(0,Y_C)$, where $u(Y_C)=\pi-C$. We may write the expression of $u_y$ as
\[
u_y=\sqrt{
\frac{(\sin u)^{2(n-1)}}{(\sin C)^{2(n-1)}}-1
}.
\]
We can also obtain the information of $u$ by looking at the inverse function $v$, which satisfies the equation
\begin{equation}\label{eq:ODE for minimal-horizontal}
    \frac{v_{xx}}{1+(v_x)^2}+(n-1)\frac{\cos(x)}{\sin(x)}v_x=0,
\end{equation}
where $v(C)=0$ and $v(\pi-C)=Y_C$, and $v_x\to \infty$ as $x\searrow C$ or $x\nearrow (\pi-C)$. dividing both sides of \eqref{eq:ODE for minimal-horizontal} by $v_x$ yields the identity
\[
\left[\log(v_x)-\frac12\log(1+(v_x)^2)+(n-1)\log\sin x\right]_x=0.
\]
Therefore, there exists $\bar\lambda_C>0$ such that
\[
\frac{v_x}{\sqrt{1+(v_x)^2}}=\bar\lambda_C(\sin x)^{-(n-1)}.
\]
Let $x\searrow C$, we have $\bar\lambda_C=(\sin C)^{-(n-1)}$. Then we obtain
\begin{equation}\label{eq:catenoid-v_x}
    v_x
=
\left(\left(\frac{\sin x}{\sin C}\right)^{2({n-1})}-1\right)^{-1/2}
\end{equation}

This implies that $v_x>0$ for all $x\in(C,\pi-C)$, and $v_x(x)+(v(\pi-x))_x=0$. Then we can show that $v(x)=2v(\pi/2)-v(\pi-x)$. It is also clear that $v(\pi/2)=Y_C/2$.
\begin{lem}\label{lem:catenoid midpoint asymptotics}
    We have the following asymptotic of $Y_C$:
    \begin{equation}\label{eq:catenoid midpoint asymptotics}
\lim_{C\to 0} Y_C=0 , \quad
               \limsup_{C\to \pi/2} Y_C \geq 2\sqrt{\frac{2}{n-1}}.
    \end{equation}
\end{lem}

\begin{proof}
By \eqref{eq:catenoid-v_x} and $v(C)=0$, we have
    \[
    \begin{split}
       Y_C=& v(\pi-C)-v(C)
        =
        \int_{C}^{\pi-C}v_x(z)dz
        \\
        =&
       \int_C^{\pi-C}
        \left(\left(\frac{\sin z}{\sin C}\right)^{2(n-1)}-1\right)^{-1/2} dz = 2\int_C^{\frac{\pi}{2}}
        \left(\left(\frac{\sin z}{\sin C}\right)^{2(n-1)}-1\right)^{-1/2} dz.
    \end{split}
    \]

Using the fact $\sin^2 z - \sin^2 C = \sin(z+C)\sin(z-C)$, for $C \in (0,\frac{\pi}{6})$, we have
\begin{equation*}
\begin{aligned}
&\int_C^{\frac{\pi}{2}}
\left(\left(\frac{\sin z}{\sin C}\right)^{2(n-1)}-1\right)^{-1/2} dz \leq \int_C^{\frac{\pi}{2}}
\left(\left(\frac{\sin z}{\sin C}\right)^{2}-1\right)^{-1/2} dz \\
=&\sin C \int_C^{\frac{\pi}{2}} \frac{1}{\sqrt{\sin^2 z - \sin^2 C}}dz = \sin C \int_C^{\frac{\pi}{2}} \frac{1}{\sqrt{\sin(z+C)\sin(z-C)}}dz \\
\leq& \frac{\sin C}{\sqrt{\sin (2C)}} \int_C^{\frac{\pi}{2}} \frac{1}{\sqrt{\sin(z-C)}}dz \leq \frac{\sin C}{\sqrt{\sin (2C)}} \int_C^{\frac{\pi}{2}} \sqrt{\frac{2}{z-C}}dz = \frac{2\sin C \sqrt{\pi-2C}}{\sqrt{\sin (2C)}}.
\end{aligned}
\end{equation*}
By letting $C \to 0$, we see that $\lim\limits_{C\to 0} Y_C=0$. On the other hand,
\begin{equation*}
\begin{aligned}
&(\sin z)^{2(n-1)} - (\sin C)^{2(n-1)}=(\sin z - \sin C)(\sin z + \sin C) [(\sin z)^{2(n-2)} + \cdots (\sin C)^{2(n-2)}] \\
&\leq 2(n-1)(\sin z - \sin C) (\sin z)^{2n-3} \leq 2(n-1)(\sin z - \sin C).
\end{aligned}
\end{equation*}
Thus
\begin{equation*}
\begin{aligned}
&\int_C^{\frac{\pi}{2}}
\left(\left(\frac{\sin z}{\sin C}\right)^{2(n-1)}-1\right)^{-1/2} dz = \int_C^{\frac{\pi}{2}}
\sqrt{\frac{(\sin C)^{2(n-1)}}{(\sin z)^{2(n-1)} - (\sin C)^{2(n-1)}}} dz \\
&\geq \int_C^{\frac{\pi}{2}} \frac{(\sin C)^{n-1} }{\sqrt{2(n-1)(\sin z - \sin C)}}dz = \frac{(\sin C)^{n-1} }{\sqrt{2(n-1)}} \int_C^{\frac{\pi}{2}} \frac{1}{\sqrt{\sin z - \sin C}} dz.
\end{aligned}
\end{equation*}

Since $\sin z - \sin C = \sin C \cos (z-C) + \cos C \sin (z-C) - \sin C \leq \cos C \sin (z-C)$, thus
\begin{equation*}
\begin{aligned}
\limsup\limits_{C \to \frac{\pi}{2}} Y_C &\geq \sqrt{\frac{2}{n-1}}\lim\limits_{C \to \frac{\pi}{2}}\int_C^{\frac{\pi}{2}} \frac{1}{\sqrt{\sin z - \sin C}} dz \geq \sqrt{\frac{2}{n-1}}\lim\limits_{C \to \frac{\pi}{2}} \int_C^{\frac{\pi}{2}} \frac{1}{\sqrt{\cos C \sin(z-C)}} dz \\
&\geq \sqrt{\frac{2}{n-1}}\lim\limits_{C \to \frac{\pi}{2}} \frac{\int_0^{\frac{\pi}{2}-C} \frac{1}{\sqrt{\sin z}} dz}{\sqrt{\cos C}} = \sqrt{\frac{2}{n-1}}\lim\limits_{C \to \frac{\pi}{2}} \frac{1/\sqrt{\sin (\frac{\pi}{2}-C)}}{\sin C / 2\sqrt{\cos C}} = 2\sqrt{\frac{2}{n-1}}.
\end{aligned}
\end{equation*}

\end{proof}

\begin{rmk}
Using the fact that $\sin z - \sin C = 2 \sin (\frac{z-C}{2}) \cos (\frac{z+C}{2}) \geq 2\sin (\frac{z-C}{2}) \cos (\frac{\pi + 2C}{4})$, by a similar estimate, we can show that $\limsup\limits_{C \to \frac{\pi}{2}} Y_C \leq 4\sqrt{\frac{2}{n-1}}$.
\end{rmk}

In fact, the spherical catenoids are periodic in the $\R$ factor, and $Y_C$ is half of the period. See Figure \ref{fig2} for a picture.

\subsection{\texorpdfstring{$\lambda$}{Lg}-Angenent curves}
In his pioneer work \cite{Angenent92_Doughnut}, Angenent constructed a rotationally symmetric self-shrinking mean curvature flow in $\R^{n+1}$ that is topological $S^1\times S^{n-1}$ before it becomes singular. More precisely, Angenent constructed a closed embedded hypersurface $\Sigma_n\subset\R^{n+1}$ that is rotationally symmetric, such that $\{\sqrt{-t}\Sigma_n\}_{t<0}$ is a mean curvature flow.

It was observed by Huisken \cite{Huisken86} that if $\{\sqrt{-t}\Sigma\}_{t<0}$ is a mean curvature flow, then $\Sigma$ satisfies the equation $\vec{H}+x^\perp/2=0$, and such $\Sigma$ is called a {\bf shrinker}. Moreover, Huisken observed that a shrinker in $\R^{n+1}$ is the critical point of the Gaussian area functional $\mathcal{F}(\Sigma)=\int_\Sigma e^{-|x|^2/4}d\mathcal{H}^n(x)$. Therefore, Angenent reduced the existence of a rotationally symmetric shrinker to the existence of a closed embedded geodesic in the half-plane $\{(x,y)|x\geq 0\}$ equipped with the metric $x^{2(n-1)}e^{-(x^2+y^2)/2}(dx^2+dy^2)$.

Angenent used a shooting method to construct such a closed embedded geodesic. The idea is to examine the trajectory of the geodesics starting from $(r,0)$ with unit tangent vector $(0,1)$ as $r$ varies. Then an interpolation argument asserts the existence of a geodesic trajectory that will get back to some $(s,0)$ again with tangent vector $(0,-1)$. Reflecting this geodesic trajectory along the $x$-axis gives a desired closed embedded geodesic.

We would like to point out that if we replace $(n-1)$ by some $\lambda>0$, all the proofs in \cite{Angenent92_Doughnut} still work (See Appendix \ref{Angenent curve} for an explanation). We call a closed embedded geodesic in the half-plane $\{(x,y)|x\geq 0\}$ equipped with the metric $x^{2\lambda}e^{-(x^2+y^2)/2}(dx^2+dy^2)$ a {\bf $\lambda$-Angenent curve}. Such curves will be the barriers. In fact, when $n\geq 3$, we can just use the $(n-1)$-dimensional Angenent torus as the barrier. However, there is no $1$-dimensional Angenent torus in the plane, so in the case $n=2$ we need a $\lambda$-Angenent curve as the barrier for $\lambda\in(0,1)$.

\section{Main result}

Throughout this paper, to simplify the notation, we will use the names of the hypersurfaces (minimal spheres, geodesic hyperspheres, spherical catenoids, etc.) to denote their profile curves and section curves as well.

\subsection{Rotationally symmetric mean curvature flow}

In this section, we study the solutions to the equation \eqref{mcf equation}. 

We will introduce the following notations throughout this section. $(M(t))_{t \in [0,T)}$ is a rotationally symmetric mean curvature flow in $S^n \times [-1,1]$. The section curve of $M(t)$ is $\gamma_t$, and can be expressed as the graph of $f(\cdot, t)$. Therefore $f(\cdot, t)$ is a solution to the equation \eqref{mcf equation}. $u(\cdot, t)$, $v(\cdot, t)$ will denote the vertical graph function and the horizontal graph function of $f(\cdot, t)$.

\begin{prop} \label{consistent}
If the initial section curve $\gamma_0$ is an ascending section curve and $T > 0$ is the first singular time, then $\gamma_t$ is an ascending section curve for all $t \in (0, T)$.
\end{prop}

\begin{proof}

$\gamma_0$ is an ascending section curve means $u'(y)>0$, and $v'(x) > 0$ besides the boundary. Since $u(\cdot, t)$, $v(\cdot, t)$ are solutions to the equation \eqref{mcf equation}, we know that $u_y$ and
 $v_x$ satisfies the following equation
\begin{equation} \label{derivative pde}
\begin{aligned}
& \frac{\partial u_y}{\partial t} = \frac{(u_y)_{yy}}{1 + (u_y)^2} - 2\frac{u_y[(u_y)_y]^2}{(1+(u_y)^2)^2} + (n - 1) \frac{u_y}{\sin^2 u}, \\
& \frac{\partial v_x}{\partial t} = \frac{(v_x)_{xx}}{1+(v_x)^2} - 2\frac{v_x[(v_x)_x]^2}{(1+(v_x)^2)^2} - (n - 1) \frac{v_x}{\sin^2 x} + (n - 1) \frac{\cos x}{\sin x} (v_x)_x.
\end{aligned}
\end{equation}

By the strict maximum principle of the quasilinear equation, we claim that for $t > 0$,
\begin{equation} \label{derivative estimate}
\begin{aligned}
&u_y(y, t)>0,\ v_x(x, t) > 0.
\end{aligned} 
\end{equation}
    
\end{proof}

From now on we will assume that the initial section curve $\gamma_0$ is ascending. It follows from Proposition \ref{consistent} that $\gamma_t$ is an ascending section curve, thus the height of $\gamma_t$ (i.e. the maximum of $f(\cdot, t)$) is given by $f(\frac{\pi}{2},t)$. We prove in the next lemma that it is monotone.

\begin{lem} \label{height decreasing}
If $\gamma_0$ is an ascending section curve, then the height of $\gamma_t$ is strictly decreasing in $t$.
\end{lem}

\begin{proof}
The horizontal line $y = f(\frac{\pi}{2},t)$ is on top of the curve $\gamma_t$, and remains static under the mean curvature flow. By Proposition \ref{comparison}, for any $t < t'$,  $y = f(\frac{\pi}{2},t)$ is on top of the curve $\gamma_{t'}$, thus $f(\frac{\pi}{2}, t') \leq f(\frac{\pi}{2}, t)$. Moreover, if $f(\cdot, 0)$ is not a constant function, then the strong maximum principle yields the result.

\end{proof}

Although Lemma \ref{height decreasing} shows that the height of the section curve of the flow strictly decreases, we do not have a quantitative estimate of the decreasing rate. To obtain such a bound, we need to construct a new family of barriers.

Given $a \in (0, \frac{\pi}{2})$, Consider a smooth function $l_a \in C^{\infty}([0,\frac{\pi}{2}])$ such that 
\begin{equation*}
l_a(x) = 
\left\{
\begin{aligned}
&0 \qquad \text{for} \qquad 0 \leq x \leq \frac{a}{2}, \\
&1 \qquad \text{for} \qquad a \leq x \leq \frac{\pi}{2},
\end{aligned} 
\right.
\end{equation*}
and $l_a' (x) > 0$ for $x \in (\frac{a}{2}, a)$.

By similar arguments as in authors' previous paper \cite{chen2023mean}, there exists a smooth solution $L_a(x,t)$ to the horizontal graph equation with initial data $L_a(x, 0) = l_a(x)$, and boundary condition $\frac{\partial}{\partial x} L_a (0, t) = \frac{\partial}{\partial x} L_a (\frac{\pi}{2}, t) = 0$. We have the following lemmas which are based on the maximum principle and Sturmian theorem (see \cite[Lemma 3.4, 3.5]{chen2023mean}).

\begin{lem} \label{height barrier}

There exist a constant $\beta_a > 0$ and a time $T_a > 0$ such that $L_a(\frac{\pi}{2}, t) < 1 - \beta_a$ for all $t > T_a$.

\end{lem}

\begin{lem}
\label{vertical graph gradient upper bound}
Given $T>0$, there exists a continuous nonincreasing function $\sigma : (0, \frac{\pi}{4}] \to \R_{+}$ that only depends on the head point of the initial condition (i.e. $u(0,0)$) and $T$, such that 
\begin{align*}
0 < u_y(y,t) \leq e^{\sigma(\delta)/t}, \quad \delta = \min \{u(y,t), \frac{\pi}{2} - u(y,t)\},
\end{align*}
holds for all $0 < t < T$, $0 < y < f(\frac{\pi}{2},t)$.

\end{lem}

Next, we examine the behavior of the head point when the singularity appears during the mean curvature flow. By Lemma \ref{height decreasing}, we know the limit of the height $h := \lim\limits_{t \to T} f(\frac{\pi}{2},t)$ exists. The gradient estimate in Lemma \ref{vertical graph gradient upper bound} implies that if the height of the function $f(\cdot, t)$ tends to $0$, then the head point must tend to the boundary $x = \frac{\pi}{2}$. 

\begin{coro} \label{height 0 head 1}

If $h = 0$, then $\lim\limits_{t \to T} u(0,t) = \frac{\pi}{2}$.

\end{coro}

\begin{proof}

We prove this by contradiction. Suppose not, then there exists $0 < \e < \frac{1}{2}$ and an increasing sequence $\{t_i\}$ such that $\lim\limits_{i \to \infty} t_i = T$, $\lim\limits_{i \to \infty} u(0,t_i) < \frac{\pi}{2} - \e$.

Then $f(x,t_i)$ is well-defined on $[\frac{\pi}{2}-\e, \frac{\pi}{2} - \frac{\e}{2}]$. Let $a(t_i) = f(\frac{\pi}{2} -\e, t_i)$, $b(t_i) = f(\frac{\pi}{2} - \frac{\e}{2}, t_i)$, then by Lemma \ref{vertical graph gradient upper bound},
\begin{align*}
\frac{\e}{2} = u(b(t_i),t_i) - u(a(t_i), t_i) = \int_{a(t_i)}^{b(t_i)} u_y (y, t_i) dy \leq (b(t_i) - a(t_i)) e^{\sigma(\frac{\e}{2})/t_i}.
\end{align*}

Hence $f(\frac{\pi}{2},t_i) \geq b(t_i) \geq \frac{\e}{2} e^{-\sigma(\frac{\e}{2})/t_i} \geq \frac{\e}{2} e^{-\sigma(\frac{\e}{2})/T}$ for all $i$. Moreover $\lim\limits_{t_i\to T}f(\frac{\pi}{2},t_i)\geq \frac{\e}{2} e^{-\sigma(\frac{\e}{2})/T}>0$, which contradicts to $\lim\limits_{t \to T} f(\frac{\pi}{2},t) = 0$.

\end{proof}

On the other hand, if the limit of the height $h$ is not zero, we obtain an improved gradient estimate.

\begin{lem} \label{vertical graph gradient lower bound}
If the height $h> 0$, then for any $0 < a < h$, let $\lambda = \frac{\pi}{h - a}$, there exists a constant $\varepsilon > 0$ such that $u_y(y,t) \geq \varepsilon e^{-\lambda^2 t} \sin(\lambda (y-a))$ for all $y \in [a,h]$, $0 \leq t < T$. In addition, $u(a,t) \leq \frac{\pi}{2} - \frac{2\varepsilon}{\lambda} e^{-\lambda^2 T}$.
\end{lem}

\begin{proof}

For $0 \leq t < T$, $u(y,t)$ is well-defined on $[0,h]$ and $u_y(y,t) > 0$ for $y \in (0,h]$. We know $u$ satisfies the vertical graph equation
\begin{align*}
u_t = \frac{u_{yy}}{1+(u_y)^2} - \frac{\cos u}{\sin u}.
\end{align*}

Define $\theta (y,t) = \arctan u_y(y,t)$, then $\theta(y,t) \in (0, \frac{\pi}{2})$ for $y \in [a,h]$, and
\begin{equation*}
\begin{aligned}
& u_t = \theta_y - \frac{\cos u}{\sin u}, \qquad  \theta_t = \frac{1}{1 + (u_y)^2} (u_y)_t = \frac{1}{1 + (u_y)^2} (u_t)_y = \frac{1}{1 + (u_y)^2} \left(\theta_{yy} + \frac{u_y}{\sin^2 u}\right).
\end{aligned}
\end{equation*}

Hence $\theta_t - \frac{\theta_{yy}}{1 + (u_y)^2} > 0$ for $y \in [a,h]$. Since $\theta(y,0) > 0$ for all $y \in [a,h]$, let $\varepsilon = \min_{y \in [a,h]} \theta(y,0) > 0$, $\varphi(y,t) = \varepsilon e^{-\lambda^2 t} \sin (\lambda(y-a))$. Then $\varphi_t = \varphi_{yy} \leq 0$ on $[a,h]$, thus
\begin{align*}
\varphi_t - \frac{\varphi_{yy}}{1+(u_y)^2} = \varphi_{yy} \frac{(u_y)^2}{1+(u_y)^2} \leq 0 < \theta_t - \frac{\theta_{yy}}{1+(u_y)^2}.
\end{align*}

Therefore $\phi$ is a subsolution, and $\theta$ is a supersolution of the PDE $f_t - \frac{f_{yy}}{1+(u_y^2)} = 0$ on the interval $[a,h]$. Initially we have $\varphi (y, 0) \leq \varepsilon \leq \theta(y, 0)$, and at the two endpoints of the interval $[a,h]$, $\varphi(a, t) = 0 < \theta(a,t)$, $\varphi(h,t) = 0 < \theta(h, t)$. By the classical maximum principle, we conclude $\theta(y, t) \geq \varphi (y,t)$ for all $a \leq y \leq h, 0 \leq t < T$. Therefore
\begin{equation*}
\begin{aligned}
& u_y(y,t) = \tan \theta(y,t) \geq \theta(y,t) \geq \varphi(y,t) = \varepsilon e^{-\lambda^2 t} \sin(\lambda (y-a)), \\
& u(h,t) - u(a,t) = \int_a^h u_y(y,t) dy \geq \int_a^h \varepsilon e^{-\lambda^2 t} \sin(\lambda (y-a)) dy = \frac{2}{\lambda} e^{-\lambda^2 t}.
\end{aligned}
\end{equation*}
This implies $u(a,t) \leq \frac{\pi}{2} - \frac{2\varepsilon}{\lambda} e^{-\lambda^2 T}$.

\end{proof}

Now we are ready to describe the possible singular behaviors of the mean curvature flow with a desired hypersurface as the initial condition. Recall that we assume the initial section curve $\gamma_0$ is ascending. In \cite{chen2023mean}, we used the Evans-Spruck estimates (\cite[Corollary 5.3]{EvansSpruck92_III}, also see \cite[Page 303]{AAG}) of the gradient of the graph function of mean curvature flow in $\R^{n+1}$. We need a generalization of this graph estimate to $S^n\times\R$ in \cite{BorisenkoMiquel12_MCF-GraphWarped}.

In the following, for $r\in(0,\pi/2)$, we use $B_r$ to denote a given ball of radius $r$ in $S^n$. 

\begin{lem} \label{neighgradestimate}
    Given $r>0$, there exists $\tau=\tau(r)$ with the following significance: Suppose $f:B_r(p)\times [0,\tau]\to\R$ is a function such that the graph of $f(\cdot,t)$ is a mean curvature flow in $B_r\times\R$, then for $t \in [0,\tau]$,
    \begin{equation}
        |\nabla f(p,t)|\leq C(\tau)+C(\tau)|\nabla f(p,0)|.
    \end{equation}
\end{lem}

\begin{proof}
The proof can be found in \cite{BorisenkoMiquel12_MCF-GraphWarped}. In the following, $\dist_p$ is the distance function on $S^n$ from the point $p\in S^n$, and for $\mu<0$, $c_\mu(\rho)=\cosh(\sqrt{-\mu}\rho)$. Using the proof of \cite[Theorem 7]{BorisenkoMiquel12_MCF-GraphWarped}, we can obtain that for any $\mu<0$,
\begin{equation}
\begin{aligned}
&\left(\frac{c_\mu(r)}{-\mu}-\frac{c_\mu(\dist_p(x))}{-\mu}e^{-\mu t}\right)^2\sqrt{1+|\nabla f(x,t)|^2} \\
\leq& \left(\frac{c_\mu(r)}{-\mu}\right)^2 
e^{-(n-1)\mu t}
\sup_{\dist_p(y)\leq r}\sqrt{1+|\nabla f(y,0)|^2}.
\end{aligned}
\end{equation}
Then when $\tau$ is sufficiently small, we can choose appropriate $\mu<0$ such that $c_\mu(r)>e^{-\mu \tau}$ to derive a desired bound for $|\nabla f(p,\tau)|$.
    
\end{proof}

\begin{prop} \label{singular time}
The flow first becomes singular at time $T$ if and only if $\lim\limits_{t \to T} u(0,t) = 0$ or $\lim\limits_{t \to T} u(0,t) = \frac{\pi}{2}$. In addition, if such $T$ doesn't exist, then the mean curvature flow exists for all future time.
\end{prop}

\begin{proof}

As shown in equation \eqref{mcf equation}, the mean curvature of the hypersurfaces near the head point is given by $\frac{u_{yy}}{1+(u_y)^2}-(n-1)\frac{\cos(u)}{\sin(u)}$, thus a singularity appears when $\lim\limits_{t \to T} u(0,t) = 0$. If $\lim\limits_{t \to T} u(0,t) = \frac{\pi}{2}$, then the section curve collapses, and singularity appears at the geodesic hypersphere. Now we assume neither of the above happens, and we want to show that $T$ is not a singular time. 

We prove by contradiction and assume that a singularity appears at time $T$. By Corollary \ref{height 0 head 1}, $h > 0$. We claim that there exists $\e_1 > 0$ such that $u(0,t) > \e_1$ for all $0 \leq t < T$. Otherwise there exists a sequence $\{t_i\}$ in $[0,T)$ such that $\lim\limits_{i \to \infty} u(0,t_i) = 0$. Up to extracting from a subsequence, we can assume $t_i$ converges to $T' \in [0, T]$. Since $u(0, T) \neq 0$, $T' < T$, and a singularity appears at the origin at time $T'$, which is a contradiction.

By a similar argument, we can also assume that $u(0,t) < \frac{\pi}{2} -\e_1$ for all $0 \leq t < T$.

For any $0 < a < h$, by Lemma \ref{vertical graph gradient lower bound}, we know $u(a, t) < \frac{\pi}{2} - \frac{2\varepsilon}{\lambda} e^{-\lambda^2 T}$. Let $\e = \min \{\e_1, \frac{2\varepsilon}{\lambda} e^{-\lambda^2 T}\}$, we know $\e < u(0,t) < u(a,t) < \frac{\pi}{2} - \e$. 

We claim that $u(y,t)$ is smooth at $(y,t) \in [0, a] \times [0,T]$. This follows from $u_y(0,t) = 0$, a priori estimate for $u_y$ in Lemma \ref{vertical graph gradient upper bound}, and hence (see \cite{ladyzhenskaia1968linear}) for all higher derivatives of $u$ in the interior. Therefore the singularity can only appear on the boundary, i.e. at $(\frac{\pi}{2},h)$.

Now consider the horizontal graph function $v(x,t)$, which is defined for $\frac{\pi}{2} - \e_1 \leq x \leq \frac{\pi}{2}$, $0 < t < T$, and we can extend this function smoothly on $x$ to $[\frac{\pi}{2}, \frac{\pi}{2} + \e_1]$ by the symmetry. This function is uniformly bounded by the height of the initial condition, and it is a solution of the horizontal graph equation. Let $\tau = \tau(\frac{\e_1}{2})$, for any $p \in (\frac{\pi - \e_1}{2}, \frac{\pi + \e_1}{2})$, consider the flow of $v_{x,t}$ on $B_{\frac{\e_1}{2}}(p) \times [T-\tau, T)$. By Lemma \ref{neighgradestimate}, $\nabla v$, as well as all higher space derivatives of $v$, are uniformly bounded on the region $\{(x,t): \frac{\pi - \e_1}{2} \leq x \leq \frac{\pi + \e_1}{2},\ T - \tau < t < T\}$. Hence $v(x,t) \to v(x,T)$ uniformly in $\frac{\pi - \e_1}{2} \leq x \leq \frac{\pi + \e_1}{2}$ as $t \nearrow T$. We have also shown that $v_t(r,t)$ is uniformly bounded for $\frac{\pi - \e_1}{2} \leq x \leq \frac{\pi + \e_1}{2}$, $T - \tau < t < T$, hence $(\frac{\pi}{2},h) = (\frac{\pi}{2}, v(\frac{\pi}{2},T))$ cannot be a singular point.

\end{proof}

In the following proposition, we show that the appearance of the singularity on the boundary $\{x = 0\}$ (i.e. the rotation axis $\{\text{North pole}\} \times [-1,1]$ in the manifold) is an open condition.

\begin{prop} \label{left open}

Let $\mathscr L$ denote the set of function $f_0 \in \mathcal{S}$ such that the solution $f(\cdot, t)$ to the equation \eqref{mcf equation} with initial condition $f(\cdot, 0) = f_0$ becomes singular in finite time, and $\lim\limits_{t \to T} u(0,t) = 0$ for some $T>0$, then $\mathscr L$ is an open set in $\mathcal{S}$ with respect to the $C^1$ norm.

\end{prop}

\begin{proof}
    The idea is similar to the proof of \cite[Proposition 3.10]{chen2023mean}. We can isometrically embed $S^n\times[-1,1]$ into $\R^{n+2}$. Then the mean curvature flow $(M^0(t))_{t\in[0,T)}$ corresponds to $f_0\in \mathcal{S}$ in $S^n\times[-1,1]$ can be interpreted as a mean curvature flow with additional force in $\R^{n+2}$. We refer the readers to \cite{HirschZhu23_forcedMCF} for some properties of mean curvature flow with additional force. In particular, by the classification of rotationally symmetric self-shrinkers, the tangent flow at the singular point of $M^0(t)$ is modeled by a cylinder $S^{n-1}(\sqrt{2(n-1)})\times\R$. That says, if we dilate the half-space by $\sqrt{T-t}$, then for any $R>0$, the profile curve of $M^0(t)$ will converge to the straight line $\{x = \sqrt{2(n-1)}\}$ as $t\to T$ in $B_R$.

    By the smooth dependence of the initial data, for any $t_0\in(0,T)$, if $f_1\in \mathcal{S}$ is sufficiently close to $f_0$ in $C^1$, then the mean curvature flow $M^1(t)$ will be very close to $M^0(t)$ on $[0,t_0]$, in particular, also very close to the straight line $\{x = \sqrt{2(n-1)}\}$ in $B_R$. Then using the appropriate dilation of the $\lambda$-Angenent curve as a barrier, we see that $M^1(t)$ can only have singularity on the rotation axis. This shows the openness of $\mathscr L$.
\end{proof}

\subsection{Interpolation family of hypersurfaces} \label{interpolation surfaces} In this section, we construct an immortal flow using the interpolation argument. By Lemma \ref{lem:catenoid midpoint asymptotics}, there exists $C_1 \in (0, \frac{\pi}{2})$ such that $Y_{C_1} = \frac{2}{n}$. Then for the spherical catenoid constructed in Section \ref{minisurface barrier} with parameter $C_1$, the restriction of its profile curve onto the region $D$ has a connected segment that joins the head point $(C_1,0)$ and $(\frac{\pi}{2}, \frac{1}{n})$. We denote this segment of the profile curve by $\mathcal{C}$.

Next, we introduce two constants $\alpha, \beta$. The $\lambda-$Angenent curve is constructed in Appendix \ref{Angenent curve}, we will let $\lambda = \frac{(n-1)}{2}$ in the following. We can choose a constant $\alpha \in (0, \frac{C_1}{4})$ and dilate the $\lambda-$Angenent curve such that the upper half of the dilated $\lambda-$Angenent curve is contained in the region $(2\alpha, \frac{
\pi}{4}) \times [0, \frac{1}{2n}]$, and we denote this segment of curve as $\mathcal{A}$, and let $T_0$ be the time at which $\mathcal{A}$ collapses to a point under mean curvature flow. We choose a constant $\beta \in (C_1, \frac{\pi}{2})$ such that $\mathcal{C}$ is on top of the curve $\{(\beta,y)| y \in [0,\frac{1}{2n}]\} \cup \{(x, \frac{1}{2n})| x \in [\beta, \frac{\pi}{2}]\}$.

We consider a family of initial data and solutions as follows:
\begin{definition}\label{def:initial}
    For $\delta\in(0,\frac{\pi}{2})$, we define a foliation of curves $\rho_\delta$ that is the graph of an increasing smooth function $f_{\delta}(x)$, with vertical graph function $u_{\delta}(y)$ and horizontal graph function $v_{\delta}(x)$, where all odd order derivatives of $u_{\delta}$ at $0$ and $v_{\delta}$ at $\frac{\pi}{2}$ vanish. We also require $\rho_{\delta}$ has the following properties:
\begin{equation} \label{starting point}
f_{\delta} \text{ is defined on the interval } [\delta, \frac{\pi}{2}], \text{ and }f_{\delta}(\delta) = 0.
\end{equation}
\begin{equation} \label{rhoalpha}
\rho_{\alpha} \text{ ends at } (\frac{\pi}{2}, \frac{1}{2n}), \text{ and } \rho_{\alpha} \text{ is on top of } \mathcal{A}.
\end{equation}
We also let $f_{\delta}(\cdot, t)$ be the family of solutions to the equation \eqref{mcf equation} with initial data $f_{\delta}(x)$. Denote the vertical graph function and the horizontal graph function of $f_{\delta}(\cdot, t)$ by $u_{\delta}(\cdot, t), v_{\delta}(\cdot, t)$, and finally we let $\rho_{\delta}(\cdot,t)$ to be the graph of $f_\delta(\cdot,t)$.

\begin{figure}[htb]
    \centering
    \tikzset{global scale/.style={
    scale=#1,
    every node/.append style={scale=#1}
  }
}
    \begin{tikzpicture}[global scale = 0.7]
    \draw (0,0) -- (9.425,0) node[below] {$(\frac{\pi}{2},0)$};
    \draw (0,0) -- (0,6) node[left] {$(0,1)$};
    \draw (9.425,0) -- (9.425,6);
    \draw (0,6) -- (9.425,6);
    \node[below] at (-0.2,0) {$(0,0)$};
    \node[below] at (9.925,6.4) {$(\frac{\pi}{2},1)$};
    \node[below] at (2.5,0) {$(\alpha,0)$};
    \node[right] at (9.425,4) {$(\frac{\pi}{2},\frac{1}{2n})$};
    \node[below] at (3.5,0) {$(\delta,0)$};    
    \node[below] at (6,0) {$(\beta,0)$};
    \draw plot[smooth,tension=.5]
coordinates {(0.5,0) (0.5,0.8) (0.6,1.5) (0.7,2.2) (1,2.8)(2,3.7) (3,4.2) (4,4.6) (5,4.8) (6,5)(7,5.1) (8,5.15) (9,5.2)(9.425,5.2)};
    \draw plot[smooth,tension=.5]
coordinates {(1.5,0) (1.5,0.6) (1.6,1.2) (1.7,1.7) (2,2.5) (3,3.5) (4,4) (5,4.3) (6,4.5)(7,4.6) (8,4.66) (9,4.7)(9.425,4.7)};
    \draw plot[smooth,tension=.5]
coordinates {(2.5,0) (2.5, 0.5) (2.6, 1) (2.7,1.3) (3,2) (4,3) (5,3.5) (6,3.7)(7,3.85) (8, 3.95) (9,4) (9.425,4)};
    \node[below] at (9,3.9) {$\rho_\alpha$};   
    \draw plot[smooth,tension=.5]
coordinates {(3.5,0) (3.5, 0.5) (3.6, 0.9) (3.7,1.3) (4,1.8) (5,2.5) (6,2.9) (7,3.1) (8,3.2) (9,3.22) (9.425,3.22)};
    \node[below] at (9,3.22) {$\rho_\delta$};  
    \draw plot[smooth,tension=.4]
coordinates {(4.5,0) (4.5,0.5) (4.55,0.9) (4.6,1.2) (4.7,1.4) (5,1.7) (5.5,1.93)(6,2.1) (7,2.3) (8,2.4) (9,2.5) (9.425,2.5)};
    \draw plot[smooth,tension=.9]
coordinates {(6,0) (6,0.3) (6.2,0.8) (6.5,1.2) (7,1.6) (7.5,1.8) (8,1.9) (9,2) (9.425,2)};
    \node[below] at (9,1.9) {$\rho_\beta$};  
    \draw
    plot[smooth,tension=.5]
coordinates {(7.5,0) (7.5,0.3)(7.6,0.6) (8,0.9) (8.5,1.1) (9,1.2)(9.425,1.2)};    
    \draw plot[smooth,tension=.5]
coordinates {(8,0) (8,0.2)(8.1,0.4) (8.5,0.6) (9,0.7) (9.425,0.7)};  
    \end{tikzpicture}
    \caption{The family of interpolation curves.}
\end{figure}
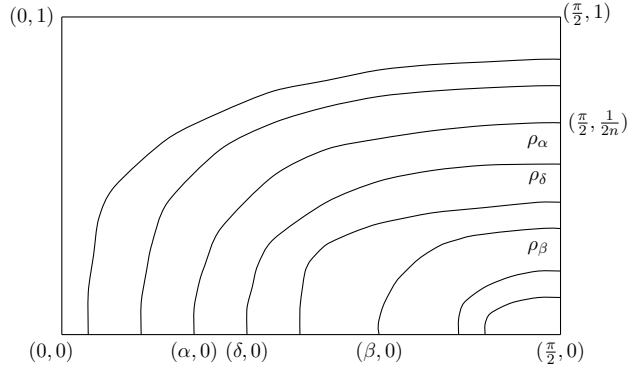

\end{definition}
After reflecting $\rho_{\delta}$ with respect to the lines $\{x = \frac{\pi}{2}\}$ and $\{y = 0\}$, we get a smooth curve in  $[0,\pi] \times [-1,1]$. By property \eqref{starting point} and the fact that $\{\rho_{\delta}\}_{\delta \in (0,\frac{\pi}{2})}$ forms a foliation, we know for $\delta > \alpha$, the curve $\rho_{\alpha}$ is on top of $\rho_{\delta}$. By property \eqref{rhoalpha}, the height of $\rho_{\alpha}$ is $\frac{1}{2n}$, since $\beta > \alpha$, thus the height of $\rho_{\beta}$ is at most $\frac{1}{2n}$. Therefore by property \eqref{starting point} and the choice of $\beta$, $\mathcal{C}$ is on top of $\rho_{\beta}$.

\begin{lem} \label{away from boundary}
Given $\delta > \alpha$, if $\lim\limits_{t \to T} u_{\delta}(0,t) = 0$, then $u_{\delta}(0,t) \leq \beta$ for all time before the singularity occurs.
\end{lem}

\begin{proof}

Since $\delta > \alpha$, we know the height of $\rho_{\delta}$ is less than $\frac{1}{2n}$. By Lemma \ref{height decreasing}, we know $f_{\delta}(\cdot, t) \leq \frac{1}{2n}$ for all time $t$.

If $u_{\delta}(0,t_1) > \beta$ at some time $t_1$ before the singularity occurs, then by the choice of $\beta$, $\mathcal{C}$ is on top of the graph of $f_{\delta}(\cdot, t_1)$. Then by Theorem \ref{comparison}, $\mathcal{C}$ is on top of the graph of $f_{\delta}(\cdot, t)$ for all time $t > t_1$, which implies that $u_{\delta}(0,t) > C_1$ for all $t > t_1$. We get a contradiction with the assumption $\lim\limits_{t \to T} u_{\delta}(0,t) = 0$.

\end{proof}

We are now ready to prove the existence of an immortal flow on $S^n \times [-1,1]$. We first sketch the picture of the two main comparison results that are used to apply the interpolation argument, See Figure \ref{fig2}. We compare the family of initial data $\rho_{\delta}$ to the curves $\mathcal{A}$ and $\mathcal{C}$. The head point of the initial curve in the first picture will tend to $0$ in finite time, meanwhile, the initial curve in the second picture will remain $C_1$ away from $0$ before the singular time.

\begin{figure}[htbp]
    \centering
    \includegraphics[width=0.7\textwidth]{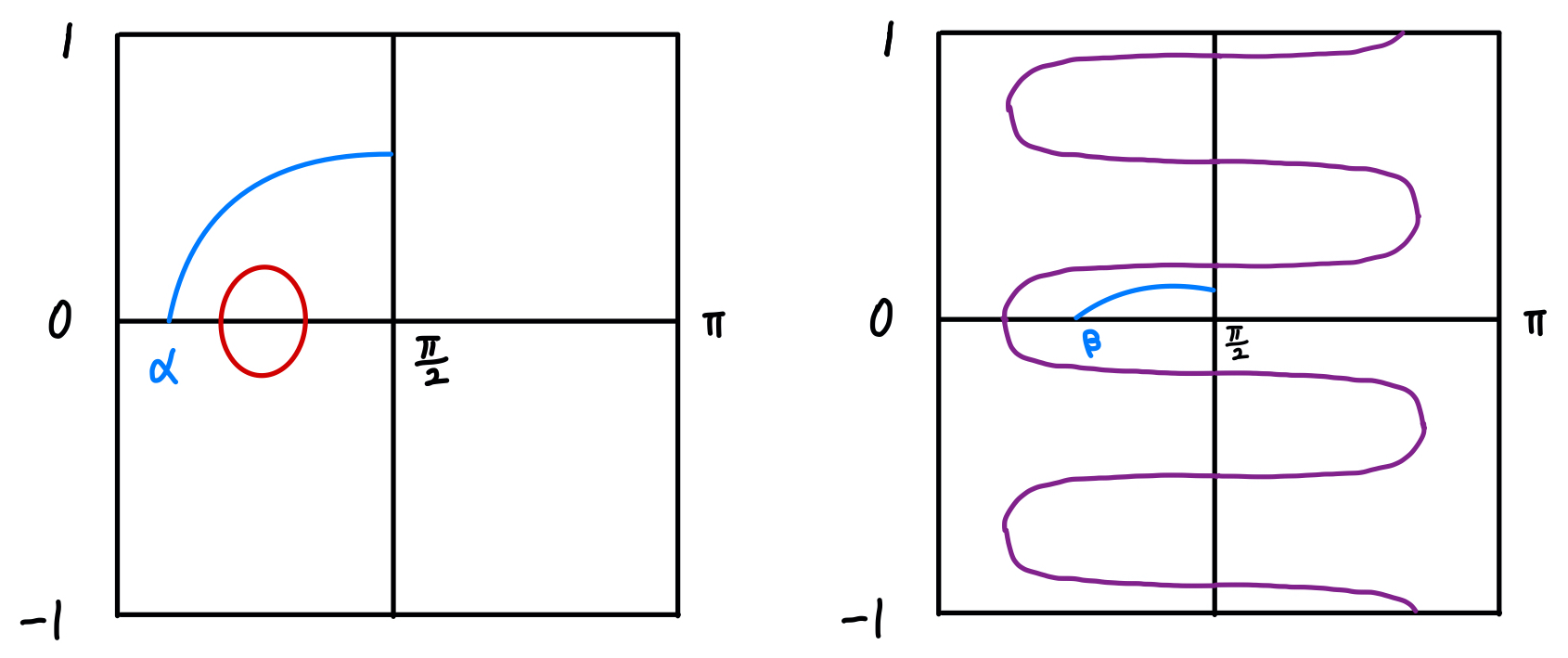}
    \caption{The blue curves are the flows that we want to study. The red curve on the left represents one $\lambda$-Angenent curve; the purple curve on the right represents one spherical catenoid.}
    \label{fig2}

\end{figure}

\begin{thm} \label{free boundary immortal flow}
There exists an immortal nonstatic rotationally symmetric mean curvature flow of hypersurfaces in $S^n \times [-1,1]$.
\end{thm}

\begin{proof}

For $\delta \leq \alpha$, we know that the curve $\rho_{\delta}$ is on top of $\mathcal{A}$. We restricted the curve $\rho_{\delta}$ and $\mathcal{A}$ onto the region $[0, \frac{\pi}{4}] \times [0,1]$, and view these two segments as the graph of a function of $y$. In this region, the function value is bounded above by $\frac{\pi}{4}$, thus
\begin{equation*}
\frac{u_{yy}}{1+(u_y)^2}-(n-1)\frac{\cos(u)}{\sin(u)} \leq \frac{u_{yy}}{1+(u_y)^2} - \frac{(n-1)}{2} \frac{1}{u}.
\end{equation*}
Hence $\sqrt{T_0 - t} \mathcal{A}$ (i.e., the solution to \eqref{lambdaAtorus} with initial condition $\mathcal{A}$) is a supersolution to the vertical graph equation, and initially $\mathcal{A}$ is on top of the curve $\rho_{\delta}$. By maximum principle, $\sqrt{T_0 - t} \mathcal{A}$ remains on top of the graph of $u_{\delta}(\cdot, t)$ for all time $t > 0$. The head point of the $\lambda-$Angenent curve $\sqrt{T_0 - t} \mathcal{A}$ tends to $0$ in finite time $T_0$, which forces $u_{\delta}(0, t) \to 0$ at some time $t < T_0$. By Proposition \ref{comparison}, if $u_{\delta}(0, t) \to 0$ in finite time, then $u_{\delta'}(0,t) \to 0$ in finite time for all $0 < \delta' < \delta$.

For $\beta \leq \delta < \frac{\pi}{2}$, $\mathcal{C}$ is on top of the curve $\rho_{\delta}$. By Proposition \ref{comparison}, $\mathcal{C}$ is on top of the graph of $f_{\delta}(\cdot, t)$ for all time $t$, which implies $u_{\delta}(0,t) \geq \beta$ for all $t$.

Hence, by continuous dependence of the initial data (see e.g. \cite{Amann}) and Proposition \ref{left open}, there exists a maximal interval $(0, \eta)$ such that for any $\delta$ within this interval, $u_{\delta}(0, t)$ converges to $0$ in finite time. As indicated in the preceding argument, $\alpha < \eta \leq \beta$. By Proposition \ref{comparison} and Proposition \ref{singular time}, the singular time $T(\rho_{\delta})$ is strictly increasing in $\delta \in (0, \eta)$, and its limit as $\delta\to \eta$ has to be $\infty$, otherwise $u_{\eta}(0,t)$ will converge to $0$ in finite time by the continuous dependence. By the selection of $\eta$, $u_{\eta}(0, t)$ never reaches $0$ in finite time. Combining this with Lemma \ref{away from boundary}, we conclude that for $\delta \in (\alpha, \eta)$ and $t \in [0, T(\rho_{\delta}))$, we have $u_{\delta}(0, t) \leq \beta$. By continuous dependence and the fact that $\lim\limits_{\delta \to \eta} T(\rho_{\delta}) = \infty$, we can deduce that $u_{\eta}(0, t) \leq \beta$ for all $t$.

As a result, $u_{\eta}(0,t)$ does not converge to either $0$ or $\frac{\pi}{2}$ within any finite time, by Proposition \ref{singular time}, the solutions $u_{\eta}(\cdot, t)$ and $v_{\eta}(\cdot, t)$ exist for all time $t \in [0, \infty)$. 

\end{proof}

\subsection{Multiplicity \texorpdfstring{$2$}{Lg} convergence}
From the construction above, we know that the solution $f_{\eta}(\cdot, t)$ to the equation \eqref{mcf equation} exists for all time $t \in [0, \infty)$. In the following, we will show the mean curvature flow induced from $f_{\eta}(\cdot,t)$ will converge to the minimal sphere $S^n \times \{0\}$ with multiplicity $2$. In fact, outside a neighborhood of the north pole and south pole of $S^n\times\{0\}$, we can show the mean curvature flow has two connected components, and each component converges to $S^n\times\{0\}$ outside a neighborhood of the north pole and south pole smoothly.

By a similar height decreasing argument as in \cite[Lemma 3.15]{chen2023mean}, we show that the head point of the function $f_{\eta}$ converges to $0$ as $t \to \infty$, i.e. $\lim\limits_{t \to \infty} u_{\eta} (0, t) = 0$. We can also use $c\mathcal{A}$ as a barrier to obtain a linear upper bound of the function $f_{\eta}$.

\begin{lem} \label{height upper bound}
$f_{\eta}(x, t) < \frac{1}{4n\alpha} x$ for all $t \in [0,\infty)$ and $ x \in (u_{\eta}(0,t), 4\alpha]$.
\end{lem} 

\begin{proof}

We prove this by contradiction. Suppose not, then $f_{\eta}(x_0,t) \geq \frac{1}{4n\alpha} x_0$ for some $t \in [0,\infty)$, $x_0 \in (u_{\eta}(0,t),4\alpha]$, then the graph of $f_{\eta}(\cdot, t)$ is on top of the dilated $\lambda-$Angenent curve $\frac{x_0}{2\alpha} \mathcal{A}$, which implies that $f_{\eta}$ has a finite time singularities. This is a contradiction.
\end{proof}

We are ready to prove the key long-time gradient estimate of $f_{\eta}$.

\begin{prop} \label{gradient estimate}
Given $0 < a < b < \frac{\pi}{2}$, for any $0 < \e_1 < 1 -  \frac{\ln (\tan \frac{b}{2})}{\ln (\tan\frac{a}{2})}$, there exist a constant $\omega$ depending on $a, \e_1$, and a time $T_1$ depending on $a,b,\e_1$, such that $$\frac{\partial}{\partial x} f_{\eta}(x,t) \leq \tan \left(\frac{\pi \ln (\tan \frac{b}{2})}{2 \ln (\tan (\frac{a}{2}))} + \frac{\omega \cos b}{\ln t} + \e_1\right),$$ for all $x \in [b,\frac{\pi}{2}]$, $t \geq T_1$. In addition, $\lim\limits_{t \to \infty} f_{\eta}(\frac{\pi}{2},t) - f_{\eta}(b,t) = 0$.
\end{prop}

\begin{proof}

Given any $0 < a < b < \frac{\pi}{4}$. Since $\lim\limits_{t \to \infty} u_{\eta} (0, t) = 0$, we know that there exists $T > 10$ such that for any $t \geq T$, $u_{\eta}(0,t) < \frac{a}{2}$. From the relations between $u_\eta$ and $f_\eta$, see Definition \ref{def:initial}, $f_{\eta}(x,t)$ is a smooth function over $[a, \frac{\pi}{2}]$ for all $t \geq T$. For simplicity, we will use $f(x, t)$ to express the function $f_{\eta}(x,t)$ restricted on the interval $[a, \frac{\pi}{2}]$, which is smooth for all time $t \geq T$. We know
\begin{align*}
f_t = \frac{f_{xx}}{1+f_x^2} + (n-1) \frac{\cos x}{\sin x} f_x, \qquad f_x > 0 \text{ for all } x \in [a,\frac{\pi}{2}).
\end{align*}

Let $\phi(x,t) = \arctan (f_x(x,t))$, $x \in [a,\frac{\pi}{2}]$, $t \in [T, \infty)$. Then $0 \leq \phi(x,t) < \frac{\pi}{2}$, and $f_x(x,t) = \tan (\phi(x,t))$. Then
\begin{align*}
\phi_x(x,t) = \frac{f_{xx}(x,t)}{1+f_x^2(x,t)}, \qquad f_t = \phi_x + (n-1)\frac{\cos x}{\sin x} f_x.
\end{align*}

Moreover, $$\phi_t = \frac{1}{1+f_x^2} (f_x)_t = \frac{1}{1+f_x^2} (f_t)_x = \frac{1}{1+f_x^2} (\phi_{xx} + (n-1) \frac{\cos x}{\sin x} f_{xx} - (n-1) \frac{f_x}{\sin^2 x}).$$ 

Therefore,

\begin{equation} \label{phi bound}
\phi_t - (n-1)\frac{\cos x}{\sin x} \phi_x - \frac{1}{1+f_x^2} \phi_{xx} = - (n-1)\frac{f_x}{\sin^2 x(1+f_x^2)} < 0.
\end{equation}

Let $\mu = -\frac{\pi}{2 \ln (\tan(\frac{a}{2}))}$ and $\omega$ be a positive constant to be determined later. Define $\varphi(x,t) = -\mu \ln (\tan \frac{x}{2}) + \frac{\omega \cos x}{\ln t} + \e_1$. Then for $x \in [a, \frac{\pi}{2}]$, $t \geq T$, we have
\begin{equation*}
\varphi_x = -\frac{\mu}{\sin x} - \frac{\omega \sin x}{\ln t} < 0, \quad \varphi_{xx} = \frac{\mu \cos x}{\sin^2 x} - \frac{\omega \cos x}{\ln t}.   
\end{equation*}
Then
\begin{equation*}
\begin{aligned}
\text{If }\varphi_{xx} \geq 0, \quad &\varphi_t - (n-1)\frac{\cos x}{\sin x} \varphi_x - \frac{\varphi_{xx}}{1+f_x^2} \geq \varphi_t - (n-1)\frac{\cos x}{\sin x} \varphi_x - \varphi_{xx} \\
&\geq \frac{-\omega \cos x}{t (\ln t)^2} + \frac{(n-1) \mu \cos x}{\sin^2 x} + \frac{(n-1) \omega \cos x}{\ln t} - \frac{\mu \cos x}{\sin^2 x} + \frac{\omega \cos x}{\ln t} \\
&= (n-2) \frac{\mu \cos x}{\sin^2 x} + \frac{\omega \cos x}{\ln t} (n - \frac{1}{t \ln t}) \geq 0. \\
\text{If }\varphi_{xx} < 0, \quad &\varphi_t - (n-1)\frac{\cos x}{\sin x} \varphi_x - \frac{\varphi_{xx}}{1+f_x^2} \geq \varphi_t - (n-1)\frac{\cos x}{\sin x} \varphi_x \\
&\geq \frac{-\omega \cos x}{t (\ln t)^2} + \frac{(n-1) \mu \cos x}{\sin^2 x} + \frac{(n-1) \omega \cos x}{\ln t}  \\
&= (n-1) \frac{\mu \cos x}{\sin^2 x} + \frac{\omega \cos x}{\ln t} (n - 1 -  \frac{1}{t \ln t}) \geq 0.
\end{aligned}
\end{equation*}

Since $f_x(\frac{\pi}{2},T) = 0$, we have $\phi(\frac{\pi}{2},T) = 0$, and there exists $0< \delta < \frac{\pi}{2}-b$ such that $\phi(x,T) < \e_1$ for all $x \in [\frac{\pi}{2} - \delta, \frac{\pi}{2}]$. Let $\omega = \frac{\pi \ln T}{2 \sin \delta}$, then for $x \in [\frac{\pi}{2} - \delta, \frac{\pi}{2}]$, $\varphi(x,T) \geq \e_1 > \phi(x,T)$; for $x \in [a, \frac{\pi}{2} - \delta)$, $\varphi(x,T) \geq \frac{\omega \cos x}{\ln T} \geq \frac{\pi}{2} \geq \phi(x,T)$.

Thus 
\begin{align*}
\begin{aligned}
& \varphi(x,T) \geq \phi(x,T), \\
&\varphi_t - (n-1) \frac{\cos x}{\sin x} \varphi_x - \frac{\varphi_{xx}}{1+f_x^2} \geq 0 > \phi_t - (n-1) \frac{\cos x}{\sin x} \phi_x - \frac{\phi_{xx}}{1+f_x^2} , \\
&\varphi(a,t) \geq -\mu \ln (\tan \frac{a}{2}) \geq \frac{\pi}{2} \geq \phi(a,t), \quad \varphi(\frac{\pi}{2},t) > 0 = \phi(\frac{\pi}{2},t).
\end{aligned}
\end{align*}

Now we can apply the parabolic maximum principle to conclude that $\varphi (x,t) \geq \phi(x,t)$ for all $x \in [a, \frac{\pi}{2}]$ and $ t \geq T_1$. On the other hand, for any $x \in [b, \frac{\pi}{2}]$,
\begin{equation*}
\varphi(x,t) \leq \varphi(b,t) = \frac{\pi \ln (\tan \frac{b}{2})}{2 \ln (\tan\frac{a}{2})} + \frac{\omega \cos b}{\ln t} + \e_1. 
\end{equation*}Thus there exists $T_1 > T$ such that for all $x \in [b, \frac{\pi}{2}]$ and $t \geq T_1$, $\varphi(x,t) < \frac{\pi}{2}$. In summary, for $x \in [b,\frac{\pi}{2}]$, $t \geq T_1$, we have the gradient estimate
\begin{equation}
    f_x(x,t) = \tan (\phi(x,t)) \leq \tan (\varphi(x,t)) \leq \tan \left(\frac{\pi \ln (\tan \frac{b}{2})}{2 \ln (\tan (\frac{a}{2}))} + \frac{\omega \cos b}{\ln t} + \e_1 \right).
\end{equation}
Next, integrating this gradient bound yields
\begin{equation*}
f(\frac{\pi}{2},t) - f(b,t) = \int_b^{\frac{\pi}{2}} f_x(x,t) dx \leq (\frac{\pi}{2} - b) \tan \left(\frac{\pi \ln (\tan \frac{b}{2})}{2 \ln (\tan (\frac{a}{2}))} + \frac{\omega \cos b}{\ln t} + \e_1 \right).
\end{equation*}

Let $t \to \infty$, 
\begin{align*}
0 \leq \limsup\limits_{t \to \infty} [f(\frac{\pi}{2},t) - f(b,t)] \leq (\frac{\pi}{2} - b) \tan \left(\frac{\pi \ln (\tan \frac{b}{2})}{2 \ln (\tan (\frac{a}{2}))}+\e_1\right).
\end{align*}

Let $a \to 0$, \begin{align*}
0 \leq \limsup\limits_{t \to \infty} [f(\frac{\pi}{2},t) - f(b,t)] \leq (\frac{\pi}{2} - b) \tan \e_1.
\end{align*}

Since $\e_1$ can be chosen arbitrarily small, we have $
\lim\limits_{t \to \infty} [f(\frac{\pi}{2},t) - f(b,t)] = 0$.

\end{proof}

The gradient estimate yields the following two corollaries, which can be combined to show that $f_{\eta}(x,t)$ will $C^1$ converge to $0$.

\begin{coro} \label{f uniform convergence}
$f_{\eta}(x, t)$ converges to $0$ uniformly as $t \to \infty$, for all $x \in (0,\frac{\pi}{2}]$.
\end{coro}

\begin{proof}

By Lemma \ref{height decreasing}, we know $\lim\limits_{t \to \infty} f_{\eta}(\frac{\pi}{2},t)$ exists. By Lemma \ref{height upper bound}, Proposition \ref{gradient estimate}, for any $0 < b \leq 4\alpha$,
\begin{align*}
0 \leq \lim\limits_{t \to \infty} f_{\eta}(\frac{\pi}{2},t) = \lim\limits_{t \to \infty} f_{\eta}(b,t) \leq \frac{b}{4n\alpha}.
\end{align*}

Let $b \to 0$, we get $\lim\limits_{t \to \infty} f_{\eta}(\frac{\pi}{2},t) = 0$. Since $f_{\eta}(\cdot, t)$ is a strictly increasing function, thus $f_{\eta}(x,t)$ converges to $0$ as $t \to \infty$ and this convergence is uniform in $x$. 

\end{proof}

\begin{coro} \label{fx uniformly convergence}
Given $0 < b < \frac{\pi}{2}$, $\frac{\partial}{\partial x} f_{\eta}(x,t)$ converges to $0$ uniformly as $t \to \infty$, for all $x \in [b, \frac{\pi}{2}]$.
\end{coro}

\begin{proof}
For any $\e > 0$, there exists $0 < a < b,$ and $0 < \e_1 < 1 -  \frac{\ln (\tan \frac{b}{2})}{\ln (\tan\frac{a}{2})}$ such that $\frac{\pi \ln (\tan \frac{b}{2})}{2 \ln (\tan (\frac{a}{2}))} < \frac{\e}{3}$, and $\e_1 < \frac{\e}{3}$. Let $\omega$, $T_1$ be as determined in Proposition \ref{gradient estimate}. Then there exists $T' \geq T_1$ such that for $t \geq T'$, $\frac{\omega \cos b}{\ln t} < \frac{\e}{3}$. By Proposition \ref{gradient estimate}, for $x \in [b,\frac{\pi}{2}]$, $t \geq T'$,
\begin{align*}
\frac{\partial}{\partial x} f_{\eta}(x,t) \leq \tan \left(\frac{\pi \ln (\tan \frac{b}{2})}{2 \ln (\tan (\frac{a}{2}))} + \frac{\omega \cos b}{\ln t} + \e_1\right) < \tan \e.
\end{align*}

\end{proof}

Let $M(t)$ denote the desired hypersurface with the graph of $f_{\eta}(x,t)$ as its section curve.

\begin{thm} \label{free boundary multiplicity 2}
$(M(t))_{t \in [0,\infty)}$ is an immortal mean curvature flow on $S^n \times [-1,1]$, and the forward limit of this flow is the minimal sphere $S^n \times \{0\}$ with multiplicity $2$.
\end{thm}

\begin{proof}

Following from the proof of Theorem \ref{free boundary immortal flow}, we only need to prove that for any $\e > 0$, $f_{\eta}(x,t)|_{x\in[\e,\frac{\pi}{2}]}$ converges to $0$ in $C^1$. Given $\e > 0$, we know that there exists $T > 0$ such that $u_{\eta}(0,t) < \e$ for $t \geq T$. Recall from Definition \ref{def:initial} that $u_{\eta}$ is the vertical graph function of $f_{\eta}$, and $\rho_{\eta}$ represents the graph of the function $f_{\eta}$. Therefore $f_{\eta}(x, t)|_{[\e, R]}$ is well defined for all $t \geq T$. By Lemma \ref{height upper bound}, $\rho_{\eta}(\cdot, t) |_{[0, \e]}$ is contained in the rectangle $[0,\e] \times [0, \frac{\e}{4n\alpha}]$ for $t \geq T$. By Corollary \ref{f uniform convergence} and \ref{fx uniformly convergence}, we know $f_{\eta}(\cdot, t) |_{[\e, \frac{\pi}{2}]}$ $C^1$ converges to $0$ as $t \to \infty$. This completes the proof.

\end{proof}

\appendix
\section{Curves in the plane with conformal metrics}\label{Appendix:curves}
Given a $C^2$ function $\phi(x,y)$ defined on an open subset $U$ of the plane, let us consider the metric $g_\phi=e^{2\phi}(dx^2+dy^2)$, and the corresponding Levi-Civita connection $\nabla$. We present some computations related to curves in $(U,g_\phi)$ in this section. 

By standard calculations, we can write down the Christoffel symbols:
\begin{align*}
    \Gamma_{xx}^x = \pr_x\phi,
     \quad 
    \Gamma_{xx}^y = -\pr_y\phi,
    \quad
    \Gamma_{xy}^x = \pr_y\phi
    ,
    \quad
    \Gamma_{xy}^y = \pr_x\phi,
    \quad
    \Gamma_{yy}^x = -\pr_x\phi,
     \quad 
    \Gamma_{yy}^y = \pr_y\phi.
\end{align*}

Suppose $\gamma:I\to U$ is a $C^2$-curve, where $I$ can be an interval or $S^1$. We consider the coordinate expression $\gamma(s)=(x(s),y(s))$. Then $\gamma'(s)=x'\pr_x+y'\pr_y$ and the length $\ell$ of $\gamma'$ is $e^\phi\sqrt{(x')^2+(y')^2}$. Let $E_s=\ell^{-1}\gamma'$. Then the geodesic curvature vector of $\gamma$ is given by
\begin{equation}
    \begin{split}
        (\nabla_{E_s}E_s)^\perp
        =&
        \ell^{-2}(x''\pr_x+y''\pr_y)^\perp
        +
        \ell^{-2}(\pr_s\ell)(\gamma')^\perp
        \\
        &+
        \ell^{-2}((x')^2 \nabla_{\pr_x}\pr_x+x'y'\nabla_{\pr_x}\pr_y+x'y'\nabla_{\pr_y}\pr_x+(y')^2\nabla_{\pr_y}\pr_y)^\perp
        \\
        =&
        (\ell^{-2}x''+\ell^{-2}((x')^2-(y')^2)\pr_x\phi
        +2\ell^{-2}x'y'\pr_y\phi
        )\pr_x^\perp
        \\&
        +
        (\ell^{-2}y''+\ell^{-2}((y')^2
        -(x')^2)\pr_y\phi
        +2\ell^{-2}x'y'\pr_x\phi
        )\pr_y^\perp.
    \end{split}
\end{equation}

Choose the unit normal vector field $\bn_\phi=\ell^{-1}(-y'\pr_x+x'\pr_y)=e^{-\phi}\bn$, where $\bn$ is the normal vector of the curve in the Euclidean metric. This implies that 
\[
\pr_x^\perp 
=
-\ell^{-1}y'e^{2\phi}\bn_\phi,
\quad 
\pr_y^\perp
=
\ell^{-1}x'e^{2\phi}\bn_\phi.
\]

Therefore, the curvature vector $\vec\kappa_\phi$ is given by
\begin{equation}
\begin{split}
    \vec\kappa_\phi=&
    (-(\ell^{-2}x''+
    \ell^{-2}((x')^2
    -(y')^2)\pr_x\phi
    +2\ell^{-2}x'y'\pr_y\phi
    )\ell^{-1}y'e^{2\phi}
        \\&
        +
        (\ell^{-2}y''+\ell^{-2}((y')^2
        -(x')^2)\pr_y\phi+2\ell^{-2}x'y'\pr_x\phi
        )\ell^{-1}x'e^{2\phi})
        \bn_\phi\\
        =&
        (
        \ell^{-3}(x'y''-x''y')
        +e^{-2\phi}\ell^{-1}
        (
        y'\pr_x\phi
        -
        x'\pr_y\phi
        ))e^{2\phi}\bn_\phi
        .
\end{split}
\end{equation}

Note that we can interpret this as
\begin{equation}
    \vec \kappa_\phi=e^{-\phi}(\vec\kappa+\nabla^\perp \phi),
\end{equation}
where $\vec\kappa$ and $\nabla^\perp$ are the curvature and normal projection of the gradient of the curve under the Euclidean metric respectively. 

\begin{example}
    Consider $\phi=f(x)$ only depending on $x$. Suppose $f(x)$ is chosen such that the curves in $(U,g_f)$ give the area of the hypersurface that is rotating the curve in some Riemannian manifold, for example, $f(x)=(n-1)\log x$ (gives the area of rotationally symmetric hypersurfaces in $\R^{n+1}$) or $f(x)=(n-1)\log \sin x$ (gives the area of rotationally symmetric hypersurfaces in $S^{n}\times\R$). Suppose the curve is $\gamma$. 
    
    After taking the quotient of $SO(n-1)$ action, we identify the mean curvature vector as $\vec h$, and the unit normal vector as $\bn$. Then from the first variational formula, we have $\vec{\mathbf h}=e^{-2f}\vec\kappa_f$, and $\bn_f=e^{-f}\bn$. After taking the quotient of the mean curvature flow of the hypersurfaces, the profile curves satisfy the equation
    \begin{equation}
        \langle \pr_t \gamma,\bn\rangle_{\text{Euc}} \bn
        =
        \vec{\mathbf{h}}.
    \end{equation}
    Under the conformal metric, this can be expressed as
    \begin{equation}
        \langle \pr_t\gamma,\bn_f\rangle_{g}\bn_f
        =
        e^{-2f}\vec\kappa_f
    \end{equation}
    Notice that this curvature flow is different from the curve shortening flow under the metric $g$ up to the conformal factor.
    
    Now we derive the equation of the graphs under this flow. Suppose $\gamma(x,t)=(x,v(x,t))$, we obtain that
    \begin{equation}
        \begin{split}
            v_t \ell^{-1}x'\bn_f
            =
            (
        \ell^{-3}(x'y''-x''y')
        +e^{-2f}\ell^{-1}
        y'\pr_xf
        )e^{2f}\bn_f
        \end{split}
    \end{equation}
\end{example}
Plugging in $x'=1$, $x''=0$, $y'=v_x$, $y''=v_{xx}$ gives
\begin{equation}
    v_t= 
        \frac{v_{xx}}{1+(v_x)^2}
        + f_x
        v_x.
\end{equation}
Similarly, suppose $\gamma(y,t)=(u(y,t),y)$, we have
\begin{equation}
    u_t= 
        \frac{u_{yy}}{1+(u_y)^2}
        - f_x(u).
\end{equation}

\section{\texorpdfstring{$\lambda$}{Lg}-Angenent curves} \label{Angenent curve}
Let $\lambda>0$. Consider the metric $g_\phi=e^{2\phi}(dx^2+dy^2)$ with $\phi(x,y)=\lambda\log x -(x^2+y^2)/8$, defined on the right half plane $\{(x,y):x>0\}$. We call a closed embedded geodesic under the metric $e^{2\phi}(dx^2+dy^2)$ a {\bf $\lambda$-Angenent curve}. When $\lambda=(n-1)$, rotating the curves by $SO(n-1)$ action gives the Angenent doughnut in $\R^{n+1}$.

Suppose $\bar\gamma(s)=(x(s),y(s))$ is a $\lambda$-Angenent curve. Let $$\gamma(s,t)=\sqrt{-t}\bar{\gamma}(s)=(\sqrt{-t}x(s),\sqrt{-t}y(s))$$ for $t<0$. Because $\bar\gamma$ is a geodesic under the metric $e^{2\phi}(dx^2+dy^2)$, following the computations in Section \ref{Appendix:curves}, we have $\ell^{-3}(x'y''-x''y')
        +e^{-2\phi}\ell^{-1}
        (
        y'\pr_x\phi
        -
        x'\pr_y\phi
        )=0$.

    From now on we will use the Euclidean coordinates on $\R^2$. Suppose $(u(y,t),y)$ is a parametrization of a part of $\gamma(\cdot,t)$, then $u(\sqrt{-t}y(s),t)=\sqrt{-t}x(s)$. Therefore we can conclude that 
    \[
    \pr_t u-\frac{y}{2\sqrt{-t}}\pr_y u
    =
    -\frac{x}{2\sqrt{-t}},
    \quad
    \pr_y u y'=x',
    \quad
    \sqrt{-t}\pr_{yy}u(y')^2+\pr_y u y''
    =
    x''.
    \]
    $\bar\gamma$ being a geodesic implies that 
    \[
    \ell^{-3}(\pr_y uy'y''-\sqrt{-t}\pr_{yy}u(y')^3-\pr_y u y'y'')
    +
    e^{-2\phi}\ell^{-1}y'(\pr_x\phi-\pr_y u\pr_y \phi)=0,
    \]
    namely
    \begin{equation}
       \frac{\sqrt{-t}\pr_{yy}u}{1+(\pr_y u)^2}
       =
       \pr_x\phi-\pr_y u\pr_y \phi
       =
       (\frac{\lambda}{x}-\frac{x}{2})+\partial_y u\frac{y}{2}.
    \end{equation}
    In summary, we obtain the equation
    \begin{equation} \label{lambdaAtorus}
        \pr_t u
        =
        \frac{\pr_{yy}u}{1+(\pr_y u)^2}-\frac{\lambda}{u}.
    \end{equation}

Angenent proved the existence of a simple closed $\lambda$-Angenent curve in \cite{Angenent92_Doughnut}. To have the same notation as \cite{Angenent92_Doughnut}, let us temporarily change the domain of the right-half plane to the upper-half plane by a rotation, and we use the $(x,r)$ as the coordinate in the upper-half plane. Namely, we want to construct a simple closed geodesic in the upper-half plane $\{(x,r):r>0\}$ equipped with the metric $r^{2\lambda}e^{-\frac{x^2+r^2}{4}}(dx^2+dr^2)$. Although in \cite{Angenent92_Doughnut}, Angenent only discussed the cases that $\lambda=n-1$ for positive integers $n\geq 2$, the proof indeed works for general $\lambda>0$ verbatim, by replacing $n$ in \cite[Page 9-16]{Angenent92_Doughnut} by $\lambda+1$. We refer the readers to \cite{Angenent92_Doughnut} for the proof. 

It is worth noting that Drugan-Nguyen \cite{Drugan-Nguyen18_Shrinking} also constructed shrinking doughnuts using a variational method. It is still an open question that whether their construction coincides with Angenent's construction. It seems that their method can also be applied to construct $\lambda$-Angenent curves. 

\medskip

\textbf{Data Availability} Data sharing not applicable to this article as no datasets were generated or analysed during
the current study.

\textbf{Conflict of interest} On behalf of all authors, the corresponding author states that there is no conflict of interest.

\bibliography{GMT}
\bibliographystyle{alpha}

\end{document}